\documentclass{amsart}
\usepackage{srcltx}

\makeatletter
 \def\LaTeX{\leavevmode L\raise.42ex
   \hbox{\kern-.3em\size{\sf@size}{0pt}\selectfont A}\kern-.15em\TeX}
\makeatother

\newcommand{\BibTeX}{{\rm B\kern-.05em{\sc
i\kern-.025emb}\kern-.08em\TeX}}

\newtheorem{thm}{Theorem}[section]
\newtheorem{lem}[thm]{Lemma}

\theoremstyle{definition}
\newtheorem{defn}{Definition}

\numberwithin{equation}{section}

\begin{document}

\title{An approach to spectral problems on Riemannian
manifolds}

\author{Isaac Pesenson}\footnote{Pacific J. Math. 215 (2004), no. 1, 183-199}
\address{Department of Mathematics, Temple University,
Philadelphia, PA 19122} \email{pesenson@math.temple.edu}


 \begin{abstract}
 It is shown that eigenvalues of Laplace-Beltrami operators
on  compact Riemannian manifolds can be determined as limits of
eigenvalues of certain finite-dimensional operators in spaces of
polyharmonic functions with singularities. In particular, a
bounded set of eigenvalues can be determined using a space of such
polyharmonic functions with a fixed set of singularities. It also
shown that corresponding eigenfunctions can be reconstructed as
uniform limits of the same polyharmonic functions with appropriate
fixed set of singularities.
\end{abstract}

\maketitle

\section{Introduction and main results}

Given an appropriate set $x_{1}, x_{2}, ..., x_{N}$ of  points "uniformly"
distributed over a compact
manifold $M, dim M=d,$ and a natural  $k>d/2$, we construct a
$N$-dimensional
subspace of polyharmonic
functions $S^{k}(x_{1},x_{2},...,x_{N})$ with singularities at $x_{1},
x_{2},
..., x_{N}$. In other words, the space  $S^{k}(x_{1},x_{2},...,x_{N})$
is the set of solutions of the following equation

$$
\Delta ^{2k}u=\sum_{\gamma=1}^{N}\alpha _{\gamma }\delta
(x_{\gamma}), k>d/2,
$$
 where $\Delta$ is the Laplace-Beltrami operator of $M$,
$\delta (x_{\gamma})$ is the Dirac measure at the point
$x_{\gamma }$ and
$$
\alpha_{1}+\alpha_{2}+...+\alpha_{N}=0.
$$

The main result (Theorem 1.3) shows that eigenvalues of the
Laplace-Beltrami operator on $M$ can be determined in two
different ways:

1) eigenvalues on an interval $[0, \omega], \omega>0,$ can be
determined as
 limits of eigenvalues of
some finite-dimensional operators in
$S^{k}(x_{1},x_{2},...,x_{N})$ when smoothness $k$ goes to
infinity, but the set of points $x_{1}, x_{2}, ..., x_{N}$ is
fixed;

2) all eigenvalues can be determined as limits of eigenvalues of
the same finite-dimensional operators in
$S^{k}(x_{1},x_{2},...,x_{N})$ when $k$ is fixed, but the number
of points $x_{1},x_{2},...,x_{N}$ is increasing.

Technically this result is a specific realization of the Rayleigh-Ritz
method [2],[3].

The above result is based on the fact (Theorem 1.4) that
eigenfunctions of $\Delta$ can be reconstructed from their values
on appropriate sets $x_{1},x_{2},...,x_{N}$ as uniform limits of
polyharmonic functions with singularities at
$x_{1},x_{2},...,x_{N}$. A weaker result in terms of homogeneous
manifolds and $L_{2}(M)$-convergence is contained in [6].

Let $\Delta$ be the Laplace-Beltrami operator on a compact,
orientable  Riemannian manifold  $M, \dim M=d,$ with metric tensor
$g$. It is known that $\Delta$ is a self-adjoint positive definite
operator in the corresponding space $L_{2}(M)$ constructed  from
$g$. Domains of the powers
 $\Delta^{s/2}, s\in \mathbb{R},$ coincide with the Sobolev spaces
$H^{s}(M), s\in \mathbb{R}$. To choose norms on spaces $H^{s}(M),$
we consider a finite cover of $M$ by balls
$B(y_{\nu},\sigma)$ where $y_{\nu}\in M$ is the center  of the ball and
$\sigma$ is its radius. For a partition of unity
${\varphi_{\nu}}$ subordinate to the family
$\{B(y_{\nu},\sigma)\}$ we introduce Sobolev space $H^{s}(M)$ as the
completion of
$C_{0}^{\infty}(M)$ with respect to the norm
\begin{equation}
\|f\|_{H^{s}(M)}=\left(\sum_{\nu}\|\varphi_{\nu}f\|^{2}
_{H^{s}(B(y_{\nu},\sigma))}\right)
^{1/2}.
\end{equation}
The regularity Theorem for the Laplace-Beltrami operator $\Delta$
states that the norm (1.1) is equivalent to the graph norm
$\|f\|+\|\Delta^{s/2}f\|$.

 We  assume that the
Ricci curvature  $Ric$ satisfies (as a form) the inequality
\begin{equation}
Ric\geq-kg, k\geq 0.
\end{equation}

 The volume of the ball $B(x,\rho)$ will be denoted
by $|B(x,\rho)|.$
Our assumptions about
 the manifold imply that there exists a constant $b>0$ such that
\begin{equation}
b^{-1}\leq\frac{|B(x,\rho)|}{|B(y,\rho)|}
\leq b, x,y\in M, \rho<r,
\end{equation}
 where $r$ is the injectivity radius
 of the manifold. The Bishop-Gromov comparison Theorem (see[8])
implies that for any $0<\sigma<\lambda<r/2$ the following
inequality holds true
$$
|B(x,\lambda)|\leq\left(\lambda/\sigma\right)^{d}e^{(k r(d-1))^{1/2}}
|B(x,\sigma)|.
$$

In what follows we use the notation

$$R_{0}(M)=12^{d}be^{(k r(d-1))^{1/2}},$$
where $d$ is the dimension of the manifold, $r$ is the injectivity radius
 and constants $k, b$ are
from (1.2) and (1.3) respectively.

The following Covering Lemma plays an important role for the
paper.

\begin{lem}
If $M$ satisfy the above assumptions then for any $0<\rho<r/6$
there
 exists a finite set of points $\{x_{i}\}$ such that

1) balls $B(x_{i}, \rho/4)$ are disjoint,

2) balls $B(x_{i}, \rho/2)$ form a cover of $M$,

3) multiplicity of the cover by balls $B(x_{i},\rho)$ is not greater
$R_{0}(M).$

\end{lem}

We will need the following definition.
\begin{defn}
For a given $0<\rho<r/6$ we say that a finite set of points $
M_{\rho}=\{x_{i}\}$ is $\rho$-admissible if it satisfies
properties 1)- 3) from the Lemma 1.1.
\end{defn}

Given a $\rho$-admissible set $M_{\rho}, |M_{\rho}|=N,$ and a
sequence of complex numbers $\{v_{i}\}_{1}^{N},$ we consider the following
variational problem:

Find  a function $f\in H^{2k}(M), k\in\mathbb{N}, k>d-1,$ such that

1) $f(x_{i})=v_{i}, i=1,...,N,$

2) $f$ is a minimizer of the functional $u\rightarrow\|\Delta^{k}f\|$.

We show that this problem does have a unique solution.

For a $\rho$-admissible set $M_{\rho}$ and a function
$f\in H^{k}(M),$ $k$ is large enough, the solution of the above
variational
 problem that
interpolates $f$ on the set $M_{\rho}$ will be denoted by $s_{k}(f).$
In fact, the function $s_{k}(f)$ depends on the set $M_{\rho}$, but we
 hope our notation will not cause any confusion. The following
Lemma  implies in particular that the set of minimizers is linear.
\begin{lem}
The set of solutions of the variational problem is the same as the set
$S^{k}(M_{\rho})$ of all solutions of the equation
\begin{equation}
\Delta ^{2k}u=\sum_{x_{\gamma }\in M_{\rho}}\alpha _{\gamma }\delta
(x_{\gamma}), k>d/2,
\end{equation}
 where $\delta (x_{\gamma})$ is the Dirac measure at the point
$x_{\gamma }$ and
\begin{equation}
\alpha_{1}+\alpha_{2}+...+\alpha_{N}=0, |M_{\rho}|=N.
\end{equation}
\end{lem}

Elements of the set $S^{k}(M_{\rho})$ will be called polyharmonic
splines. Since zero is the simple eigenvalue of the scalar
Laplace-Beltrami operator on a compact connected manifold  and
because equation (1.4) on a compact
 manifold is solvable only under assumption (1.5) the
dimension of the space  $S^{k}(M_{\rho})$ is exactly $|M_{\rho}|=N.$

It was shown in [6] that there are functions $L_{\nu}^{k}\in S^{k}
(M_{\rho})$ such that $L_{\nu}^{k}$ takes value $1$ at $x_{\nu}$
and $0$ at all other points of $M_{\rho}.$ Moreover these
functions form a  basis of $S^{k}(M_{\rho}).$

Let $0< \lambda_{1}\leq \lambda_{2}\leq ...\leq\lambda_{j}$ be the
 sequence of the first $j$ eigenvalues of the operator $\Delta$
in $L_{2}(M)$ counted with their
multiplicities and $\varphi_{1}, \varphi_{2},...,\varphi_{j}$
is the corresponding set of orthonormal
 eigen functions. Throughout the paper $\|.\|$ denotes the $L_{2}(M)$
norm.

 According to the min-max principle
for a self-adjoint positive definite operator $D$ in a Hilbert space
 $E$ the
$j$-th eigenvalue can be calculated by the formula
$$
\lambda_{j}=inf_{F\subset E} sup_{f\in F}
\frac{\|D^{1/2}f\|_{E}^{2}}{\|f\|_{E}^{2}},f\neq 0,
$$
where $inf$ is taken over all $j$-dimensional subspaces $F$ of
$ E$.

We introduce the numbers $\lambda_{j}^{(k)}(M_{\rho})$ by the formula
\begin{equation}
\lambda_{j}^{(k)}(M_{\rho})=inf_{F\subset S^{k}(M_{\rho})} sup_{f\in F}
\frac{\|\Delta^{1/2}f\|^{2}}{\|f\|^{2}}, f\neq 0,
\end{equation}
where $inf$ is taken over all $j$-dimensional subspaces of
$ S^{k}(M_{\rho})$.

As a consequence of the min-max
 principle we obtain that
the numbers $\lambda_{j}^{(k)}(M_{\rho})$ are the eigenvalues of
the matrix $D^{(k)}=D^{(k)}(M_{\rho})$ with entries
\begin{equation}
d_{\gamma,\nu}^{(k)}=\int_{M}(\Delta L^{k}_{\gamma})
L^{k}_{\nu}dx,
\end{equation}
where $dx$ is the Riemannian density.

Now we can formulate our main result which shows that eigenvalues
of matrices $D^{(k)}$
 approximate eigenvalues of the Laplace-Beltrami operator and
the rate of convergence is exponential.

\begin{thm}
There exists a $C_{0}=C_{0}(M)$ such that for any given $\omega>0$
if $0<\rho<\left(C_{0}\omega\right)^{-1/2}$ then for every
$\rho$-admissible set $M_{\rho}$, every eigenvalue
$\lambda_{j}\leq \omega$ and all $k=(2^{l}+1)d, l=0, 1, ...,$
\begin{equation}
\lambda_{j}^{(k)}(M_{\rho})- \omega^{2d}\gamma^{2(k-d)}\leq
\lambda_{j}\leq\lambda_{j}^{(k)}(M_{\rho}) ,
\end{equation}
where $\gamma=C_{0}\rho^{2}\omega<1.$
\end{thm}

The inequality (1.8) shows that there are three different ways to
determine
eigen values $\lambda_{j}.$

1) Eigen values from the interval $[0,\omega]$  can be determined
by keeping a set $M_{\rho}$ with $0<\rho<(C_{0}\omega)^{-1/2}$
fixed and by letting $k$ go to infinity.

2) By letting $\rho$ go to zero and keeping $k$ fixed one can determine
all of the eigen values.

3) The convergence will be even faster if $\rho$ goes to zero
and at the same time $k$ goes to infinity.

The following  Approximation Theorem plays a key role in the proof
of the Theorem 1.3.

\begin{thm}
There exist constants $C(M), \rho(M)>0$ such that for any
$0<\rho<\rho(M)$, any $\rho$-admissible set $M_{\rho}$,
any smooth function $f$ and any
$t\leq d$ the following inequality holds true
\begin{equation}
\|\Delta^{t}(s_{k}(f)-f)\|\leq \left(C(M)\rho^{2}\right)^{k-d}
\|\Delta
^{k}f\|,
\end{equation}
for any $k=(2^{l}+1)d, l=0, 1, ... .$
In particular, if $f$ is a linear combination of orthonormal eigen
functions whose corresponding eigen values belong to the
interval $[0, \omega],$ then for any $t\leq d$
\begin{equation}
\|\Delta^{t}(s_{k}(f)-f)\|\leq \omega^{d}
\left(C(M)\rho^{2}\omega\right)^{k-d}\|
f\|,
\end{equation}
where $k=(2^{l}+1)d, l=0, 1, ... .$

Moreover, we have the following estimates in the uniform norm on the
manifold
$$\sup_{x\in M}|(s_{k}(f)(x)-f(x))|\leq \left(C(M)\rho^{2}\right)^{k-d}
\|\Delta
^{k}f\|, k=(2^{l}+1)d, l=0, 1, ...
$$
and respectively,
$$
\sup_{x\in M}|(s_{k}(f)(x)-f(x))|\leq \omega^{d}
\left(C(M)\rho^{2}\omega\right)^{k-d}\|
f\|,k=(2^{l}+1)d, l=0, 1, ... ,
$$
if $f$ belongs to the span of eigenfunctions whose eigenvalues are
not greater than $\omega$.
\end{thm}

Proofs of the Theorems 1.3 and 1.4 show that the constants $C_{0}(M),
C(M),\rho(M)$ depend on the bounds on the curvature of $M$.

We obtain our Approximation Theorem as a consequence of the following
inequality that is a Poincare-type inequality.

\begin{thm}
There exist $C(M), \rho(M)>0$ such that for any $0<\rho<\rho(M)$,
any $\rho$-admissible set $M_{\rho}$ and any $f\in H^{2dm}(M)$
whose restriction to $M_{\rho}$ is zero the following inequality
holds true
\begin{equation}
\|f\|\leq \left(C(M)\rho^{2d}\right)^{m}
\|\Delta^{dm}f\|,
\end{equation}
where $ m=2^{l}, l=0, 1, ....$
\end{thm}

According to our main Theorem 1.3, if we are going to determine
the spectrum on an interval $[0, \omega], \omega>0,$ we have to
use an $\rho$-admissible set of points $M_{\rho}$, where
$0<\rho<(C_{0}\omega)^{-1/2}.$ It is clear that the cardinality of
$M_{\rho}$ i. e. $N=|M_{\rho}|$ cannot be less than the number of
eigen values on the interval $[0, \omega].$ In fact if
$\rho=\epsilon \left(C_{0}\omega\right)^{-1/2}, 0<\epsilon<1,$
then the number of points in $M_{\rho}$ is approximately
\begin{equation}
N=|M_{\rho}|\asymp\frac{Vol M}{\rho^{d}}= \epsilon^{-d}
C_{0}^{d/2}Vol M\omega^{d/2}.
\end{equation}
Note, that according to the Weyl's asymptotic formula the number of
eigen values on an interval $[0,\omega]$ is asymptotically
\begin{equation}
c Vol M\omega^{d/2}.
\end{equation}
In other words our method requires an "almost" optimal number of points
for admissible sets $M_{\rho}.$

It is important to realize an interesting feature of the inequalities
(1.9)- (1.11): all the constants and the interval for admissible
$\rho$'s depend solely on the manifold, while the exponents $k$ and
$m$ can be made arbitrary large.
These inequalities are consequences of the inequality (1.11).
To obtain (1.11) we establish it first for  $m=1$ and then
apply the following result which allows to "exponentiate" right-hand sides
of some inequalities.

\begin{lem}
1) If for some $f\in H^{2s}(M), a,s>0,$
\begin{equation}
\|f\|\leq a\|\Delta^{s}f\|,
\end{equation}
then for the same $f, a, s$ and all $t\geq 0, m=2^{l}, l=0, 1, ...,$
\begin{equation}
\|\Delta^{t}f\|\leq a^{m}\|\Delta^{ms+t}f\|,
\end{equation}
if $f\in H^{2(ms+t)}(M).$

\end{lem}

As an application of our main result we prove that the
zeta-function $\zeta(s)$ of the Laplace-Beltrami operator
\begin{equation}
\zeta(s)=\sum_{\lambda_{i}\neq 0}\lambda_{i}^{-s}.
\end{equation}
is the uniform limit of zeta-functions for finite-dimensional
 operators $D^{k}.$
Namely, we choose a sequence $\eta_{n}$ that goes to zero and for every
$\eta_{n}$ construct a set $M_{\eta_{n}}.$ For a fixed $k$ that is large
enough
we consider the space $S^{k}(M_{\eta_{n}})$ and the eigen values
 of the corresponding operator $D^{(k)}$ defined by (1.7)
we denote as $\lambda_{i}^{(k)}
(\eta_{n}).$ The $\zeta$-function for a finite-dimensional operator
$D^{(k)}$ is denoted by $\zeta_{\eta_{n}}(s).$

\begin{thm}
The sequence of $\zeta$-functions $\zeta_{\eta_{n}}(s)$ converges
uniformly to $\zeta(s)$ on compact subsets of the set
$\{s=u+iv|u>d/2\}.$
\end{thm}

\section{Proof of the Theorem 1.3}

Let $P^{k}_{M_{\rho}}$ be the projector from $ H^{d/2+1}(M)$ onto
the space $S^{k}(M_{\rho})$ defined by the formula
 $P^{k}_{M_{\rho}}f=
s_{k}(f).$ Note that the function $s_{k}(f)$ depends on the set
$M_{\rho}.$

For a given $\omega>0$ let
$0<\lambda_{1}\leq\lambda_{2}\leq...\leq\lambda_{j(\omega)}
\leq\omega$ be the set of all eigen values  counted with their
multiplicities which are not greater than $\omega$. If
$\varphi_{1},\varphi_{2},...,\varphi_{j(\omega)}$ is the set of
corresponding
orthonormal eigen functions then their $span$ is denoted by $E_{\omega}$.
Note, that $dim E_{\lambda_{i}}=i.$ If $\omega \in [\lambda_{j(\omega)},
\lambda_{j(\omega)+1})$ then $E_{\omega}=E_{\lambda_{j(\omega)}}$ and
$dim E_{\omega}= dim E_{\lambda_{j(\omega)}}=j(\omega).$

According to Approximation Theorem 1.4, inequality (1.10),  for any
$\varphi_{i}$ such that
 the corresponding
$\lambda_{i}\leq \omega $ we have
$$\|s_{k}(\varphi_{i})-\varphi_{i}\|\leq\omega^{d} (C_{0}\rho^{2}
\omega)^{k-d}, s_{k}(\varphi_{i})\in S^{k}(M_{\rho}), k=d(2^{l}+1),
l=0, 1, ...$$
The right hand side in the last inequality goes to zero for
$0<\rho<(C_{0}\omega)^{-1/2}$ and large $k$.
Thus, the  dimension of $P^{k}_{M_{\rho}}(E_{\omega})$ is $j(\omega)$
as long as $0<\rho<(C_{0}\omega)^{-1/2}$ and $k$ is large enough.

Next according to the min-max principle the eigen value
$\lambda_{j}$ of $\Delta$ can be defined by the formula
$$
\lambda_{j}=inf_{F\subset L_{2}(M)} sup_{f\in F}
\frac{\|\Delta^{1/2}f\|^{2}}{\|f\|^{2}},f\neq 0,$$
 where $inf$ is taken over all $j$-dimensional subspaces of $L_{2}(M)$.

It is clear that
\begin{equation}
\lambda_{j} \leq \lambda_{j}^{(k)}(M_{\rho})\leq
sup_{f\in P^{k}_{M_{\rho}}(E_{\lambda_{j}})}
\frac{\|\Delta^{1/2}f\|^{2}}{\|f\|^{2}}, f\neq 0,
\end{equation}
where $\lambda_{j}^{(k)}$ is defined by (1.5), $\lambda_{j}\leq\omega,
0<\rho<(C_{0}\omega)^{-1/2}$ and $k$ is large enough.

For any $\psi\in E_{\lambda_{j}}$, set
$h_{k}=s_{k}(\psi)-\psi,$ and
$$
h_{k}=h_{k,j}+h_{k,j}^{\perp},
$$
where $h_{k,j}\in E_{\lambda_{j}}, h_{k,j}^{\perp}\in
E_{\lambda_{j}}^{\perp}.$

It gives

$$\Delta^{1/2}h_{k}=\Delta^{1/2}h_{k,j}+\Delta^{1/2}h_{k,j}^{\perp}.$$
Since $\Delta$ is
self adjoint and $E_{\lambda_{j}}$ is its invariant subspace the terms on
the
 right are orthogonal and we obtain

\begin{equation}
 \|\Delta^{1/2}h_{k,j}^{\perp}\|\leq
\|\Delta^{1/2}h_{k}\|.
\end{equation}

It is clear that the orthogonal projection
of $s_{k}(\psi)$ onto $E_{\lambda_{j}} $ is $\psi+h_{k,j}=\psi_{k,j}.$
 Since $s_{k}(\psi)=\psi_{k,j}+h_{k,j}^{\perp},$ we have
$$
\|s_{k}(\psi)\|^{2}\geq \|\psi_{k,j}\|^{2}
$$
and we also have

$$\Delta^{1/2}s_{k}(\psi)=\Delta^{1/2}\psi_{k,j}+
\Delta^{1/2}h_{k,j}^{\perp},
$$
that implies

$$\|\Delta^{1/2}s_{k}(\psi)\|^{2}=\|\Delta^{1/2}\psi_{k,j}\|^{2}+
\|\Delta^{1/2}h_{k,j}^{\perp}\|^{2}.$$

After all we obtain the following inequality

$$
\frac{\|\Delta^{1/2}s_{k}(\psi)\|^{2}}{\|s_{k}(\psi)\|^{2}}\leq
\frac{\|\Delta^{1/2}\psi_{k,j}\|^{2}}{\|\psi_{k,j}\|^{2}}+
\frac{\|\Delta^{1/2}h_{k,j}^{\perp}\|^{2}}{\|s_{k}(\psi)\|^{2}}.
$$

The last inequality along with inequalities (2.1) and (2.2) gives
\begin{equation}
\frac{\|\Delta^{1/2}s_{k}(\psi)\|^{2}}{\|s_{k}(\psi)\|^{2}}\leq
\lambda_{j}+
\frac{\|\Delta^{1/2}h_{k}\|^{2}}{\|s_{k}(\psi)\|^{2}}.
\end{equation}

In what follows we will use the notation
$$h_{k}^{(i)}=h_{k}^{(i)}(M_{\rho})=s_{k}(\varphi_{i})-\varphi_{i},$$
where $\varphi_{i}$ is the $i$-th orthonormal eigen
function.

According to Approximation Theorem 1.4,  $\|h_{k}^{(i)}(M_{\rho})\|$ can
be done arbitrary
small for large $k$
if corresponding eigen value $\lambda_{i}\leq\omega$ and
$0<\rho<(C_{0}\omega)^{-1/2}$ because

$$\|h_{k}^{(i)}(M_{\rho})\|
\leq\omega^{d} (C_{0}\rho^{2}\omega)^{k-d},
 k=d(2^{l}+1), l=0, 1, ... .$$

Assume that $0<\rho<\left(C_{0}\omega\right)^{-1/2}$ and $k$ is so large
that
\begin{equation}
\sum_{i=1}^{j(\omega)}\|h_{k}^{(i)}(M_{\rho})\|^{2}\leq 1/2
\end{equation}
where $j(\omega)$ is the number of all eigen values (counting with
their multiplicities) which are not greater than $\omega$.

Using the fact that $\Delta^{1/2}$ is a self adjoint operator one
can
show that
$$
\|\Delta^{1/2}h_{k}\|\leq\|\psi\|\left(\sum_{i=1}^{j(\omega)}
\|\Delta^{1/2}h_{k}^{(i)}\|^{2}\right)
^{1/2},
$$
where $h_{k}=s_{k}(\psi)-\psi.$

The last inequality and the inequality
(2.3) imply
$$
\lambda_{j}^{(k)}(M_{\rho})-\lambda_{j}\leq sup_{\psi\in E_{\lambda_{j}}}
\frac{\|\Delta^{1/2}s_{k}(\psi)\|^{2}}{\|s_{k}(\psi)\|^{2}}-\lambda_{j}\leq
$$

\begin{equation}
sup_{\psi\in E_{\lambda_{j}}}\frac{\|\Delta^{1/2}h_{k}\|^{2}}
{\|s_{k}(\psi)\|^{2}}\leq
sup_{\psi\in
E_{\lambda_{j}}}\frac{\|\psi\|^{2}\sum_{i=1}^{j(\omega)}\|\Delta^{1/2}h_{k}^{(i)}\|
^{2}}{\|s_{k}(\psi)\|^{2}}.
\end{equation}

Since
$$\|\psi\|=\|s_{k}(\psi)-h_{k}\|\leq \|s_{k}(\psi)\|+\|h_{k}\|$$
and
$$\|h_{k}\|^{2}\leq \|\psi\|^{2}
\sum_{i=1}^{j(\omega)}\|h_{k}^{(i)}\|^{2}
\leq \frac{1}{2}\|\psi\|^{2},$$
we have
$$
\|s_{k}(\psi)\|^{2}\geq \left(\|\psi\|-\|h_{k}\|\right)^{2}\geq
\frac{1}{4}\|\psi\|^{2}.$$
After  using (2.5) we obtain
$$
\lambda_{j}^{(k)}(M_{\rho})-\lambda_{j}\leq
4\sum_{i=1}^{j(\omega)}\|\Delta^{1/2}h_{k}^{(i)}(M_{\rho})\|^{2}.
$$

Because the Sobolev space $H^{s}(M)$ is continuously embedded into
the space $H^{t}(M)$ if $ s>t$, we obtain that according
to the estimates (1.9) and (2.4)
$$
\|\Delta^{1/2}h_{k}^{(i)}(M_{\rho})\|^{2}\leq\omega^{2d}
\left(C(M)\rho^{2}\omega\right)^{2(k-d)}\|h_{k}^{(i)}(M_{\rho})\|^{2}
\leq\omega^{2d}\left(C_{0}\rho^{2}\omega\right)^{2(k-d)}.
$$

Finally we have
$$
\lambda_{j}\leq\lambda_{j}^{(k)}(M_{\rho})\leq\lambda_{j}+
\omega^{2d}\left(C_{0}\rho^{2}\omega\right)^{2(k-d)},
k=d(2^{l}+1), l=0, 1,... .
$$
 where $\lambda_{j}\leq\omega, 0<\rho<(C_{0}\omega)^{-1/2}$ and $k$ is
large enough.

Theorem 1.3 is proved.

\section{Approximation of the zeta-function of the
 Laplace-Beltrami operator}

The zeta-function $\zeta(s)$ of the Laplace-Beltrami operator is defined
by formula (1.16).
Since the paper of Minakshhisundaram and Pleijel [4] it is known
that this series
 converges absolutely for every $s=u+iv$ where $u>d/2.$ As a result
it
 converges uniformly on every half-plane whose closure is a proper subset
of the set $\{s=u+iv|u>d/2\}.$

Now we choose a sequence $\eta_{n}$ that goes to zero and for every
$\eta_{n}$ construct a set $M_{\eta_{n}}.$ For a fixed $k>d/2$
we consider the space $S^{k}(M_{\eta_{n}})$ and the eigen values
 of the corresponding operator $D^{(k)}$ we denote as $\lambda_{i}^{(k)}
(\eta_{n}).$ The $\zeta$-function for a finite -dimensional operator
$D^{(k)}$ is denoted by $\zeta_{\eta_{n}}(s).$

\begin{thm}
The sequence of $\zeta$-functions $\zeta_{\eta_{n}}(s)$ converges
uniformly to $\zeta(s)$ on compact subsets of the set
$\{s=u+iv|u>d/2\}$ as $\eta_{n}$ goes to zero.
\end{thm}

\begin{proof}
Let $\Omega\subset\{s=u+iv|u>d/2\}$ be a compact set. For a fixed
$\varepsilon>0$ let $m\in \mathbb{N}$ a such integer that
$$
\sum_{j\geq m}|\lambda_{j}^{-s}|=
\sum_{j\geq m}\lambda_{j}^{-Re s}\leq \varepsilon/3
$$
for all $s\in \Omega.$

Since for every $k$, we have that $ \lambda_{j}^{(k)}(\eta_{n})
\geq\lambda_{j},$
$$
|(\lambda_{j}^{(k)}(\eta_{n})^{-s}|=(\lambda_{j}^{(k)}(\eta_{n}))^{-Re s}
\leq \lambda_{j}^{-Re s}
=|\lambda_{j}^{-s}|.
$$
Thus
\begin{equation}
\mid\sum_{j\geq m}\lambda_{j}^{-s}-\sum_{j=m}^{N(n)}(\lambda_{j}^{(k)}
(\eta_{n}))^{-s}
\mid\leq \frac{2\varepsilon}{3}.
\end{equation}
Next, because $\lambda_{j}^{(k)}(\eta_{n}) $ goes to $\lambda_{j}$ as
$\eta_{n}$ goes to zero, we can find $n=n(\varepsilon)$ such that for
 $n>n(\varepsilon), s\in \Omega$
\begin{equation}
\mid\sum_{j\leq m}\lambda_{j}^{-s}
-\sum_{j\leq m}^{N(n)}(\lambda_{j}^{(k)}
(\eta_{n}))^{-s}
\mid\leq \varepsilon/3,
\end{equation}
where we assume that $\lambda_{j}\neq 0, \lambda_{j}^{(k)}(\eta_{n})\neq
0.$

The last inequalities imply
$$
|\zeta(s)-\zeta_{\eta_{n}}(s)|\leq\varepsilon
$$
if $n>n(\varepsilon), s\in \Omega.$ Theorem is proved.

\end{proof}

\section{A Poincare type inequality and spline approximation on manifolds}

We consider a compact orientable Riemannian manifold whose Ricci
curvature satisfies (1.1). First, we prove the Covering Lemma 1.1
from Introduction (compare to [1]).
\begin{proof}

Let us choose a family of disjoint balls $B(x_{i},\rho/4)$ such
that there is no ball $B(x,\rho/4), x\in M,$ which has empty intersections
with all balls from our family. Then the family $B(x_{i},\rho/2)$ is a
cover of
$M$. Every ball from the family $\{B(x_{i}, \rho)\}$, that has
non-empty intersection with a particular ball $\{B(x_{j}, \rho)\}$ is
contained in the ball $\{B(x_{j}, 3\rho)\}$. Since any two balls from the
family $B(x_{i},\rho/4)$
are disjoint, it gives the following estimate for the index of
multiplicity
$R$ of the cover $B(x_{i},\rho)$:
\begin{equation}
R\leq\frac{\sup_{y\in M}|B(y,3\rho)|}{\inf_{x\in M}|B(x,\rho/4)|}.
\end{equation}

As it was mentioned in the Introduction, the Bishop-Gromov comparison
 theorem (see [8]) implies that for any $0< \sigma<\lambda<r/2$
\begin{equation}
|B(x,\lambda)|\leq (\lambda/\sigma)^{d} e^{(kr(d-1))^{1/2}}|B(x,\sigma)|.
\end{equation}

This property along with (4.1) allows to continue the
estimation of $R$:

$$R\leq 12^{d}e^{(kr(d-1))^{1/2}}
\frac{\sup_{y\in M}|B(y,\rho/4)|}{\inf_{x\in M}|B(x,\rho/4)|}
\leq 12^{d}b e^{(kr(d-1))^{1/2}}=R_{0}(M).
$$

\end{proof}

We will  need the following result which is in fact a global
 Poincare type inequality.

\begin{lem}
For any $k>d-1$ there exist constants $C(M,k)>0, \rho(M,k)>0$
 such that for any $\rho<\rho(M,k)$ and any $\rho$-admissible
set $M_{\rho}=\{x_{i}\}$ the following inequality holds true

\begin{equation}
\|f\|\leq C(M,k)\left\{\rho^{d/2}\left(\sum_{x_{i}\in M_{\rho}}
|f(x_{i})|^{2}\right)^{1/2}+\rho^{2k}\|\Delta^{k}f\|\right\}, k>d-1.
\end{equation}

\end{lem}
\begin{proof}

Let $M_{\rho}=\{x_{i}\}$ be a $\rho$-admissible set
and $\{\varphi_{\nu}\}$ the partition of unity from (1.1). For any $f\in
C^{\infty}(M)$, every fixed $B(x_{i},\rho)$ and every
 $x\in B(x_{i},\rho/2)$
$$
(\varphi_{\nu}f)(x)=(\varphi_{\nu}f)(x_{i})+\sum_{1\leq|\alpha|\leq n-1}
\frac{1}{\alpha !}\partial^{|\alpha|}(\varphi_{\nu}f)(x_{i})(x-x_{i})
^{\alpha}+
$$
\begin{equation}
\sum_{|\alpha|=n}\frac{1}{(n-1)!}\int_{0}^{\tau}t^{n-1}\partial
^{|\alpha|}(\varphi_{\nu}f)(x_{i}+t\vartheta)\vartheta^{\alpha}dt,
\end{equation}
where $x=(x_{1},...,x_{d}), x_{i}=(x_{1}^{i},...,x_{d}^{i}), \alpha=(
\alpha_{1},...,\alpha_{d}), x-x_{i}=(x_{1}-x_{1}^{i})^{\alpha_{1}}...
(x_{d}-x_{d}^{i})^{\alpha_{d}}, \tau=\|x-x_{i}\|,
\vartheta=(x-x_{i})/ \tau.$

We are going to make use of the following inequality.
\begin{equation}
|\partial^{|\alpha|}(\varphi_{\nu}f)(x_{i})|\leq C_{d,m}\sum_{|\mu|
\leq m}
\rho^{|\mu+\alpha|-d/2}\|\partial^{|\mu+\alpha|}(\varphi_{\nu}f)\|
_{L_{2}(B(x_{i},\rho))},
\end{equation}
where $\mu=(\mu_{1}, \mu_{2}, ... ,\mu_{d}), m>d/2.$
To prove this inequality we first recall the following inequality
(see[5]):

 there exists a constant $c_{d,m}$ such that for every
$\psi\in C_{0}^{\infty}(B(x_{i},\rho/2))$
$$
|\psi(x_{i})|\leq
c_{d,m}\rho^{m-d/2}\|\psi\|_{H^{m}(B(x_{i},\rho))},m>d/2.
$$

We consider the function
$\xi(x)=e \exp(1/(\|x\|^{2}-1))$, if $\|x\|<1$ and
 $\xi(x)=0,$ if $\|x\|\geq 1.$ It is clear that
$\xi\in C_{0}^{\infty}(U_{0})\subset \mathbb{R}^{d}, \xi(0)=1,$ where
$U_{0}$ is the unit ball of $ \mathbb{R}^{d}.$ Set $\xi_{\rho}(x)=
\xi(2\rho^{-1}(x-x_{i})).$ Since for any
$\psi\in C^{\infty}(B(x_{i},\rho))$
we have that $\xi_{\rho}\psi\in C^{\infty}_{0}(B(x_{i},\rho/2))$ and
$\xi_{\rho}\psi(x_{i})=\psi(x_{i}), $ we can use the last inequality
to obtain
$$
|\psi(x_{i})|\leq c_{d,m}\rho^{m-d/2}\|\xi_{\rho}\psi\|_{H^{m}
(B(x_{i},\rho))}, m>d/2.
$$
This inequality shows  that there exist constants $C(d,m,n)$ such
that for any $\psi\in C^{\infty}(B(x_{i},\rho))$
$$
|\psi(x_{i})|\leq \sum_{   n\leq m    } C(d,m,n)\rho^{n-d/2}\|\psi\|
_{H^{n}(B(x_{i},\rho))}, m>d/2.
$$

This inequality implies the inequality (4.5) when
$\psi=\partial^{|\alpha|}
(\varphi_{\nu}f)$.

Now we continue the estimation of the second term in (4.4).
$$
\int_{B(x_{i},\rho/2)}|\sum_{1\leq|\alpha|\leq n-1}
\frac{1}{\alpha !}\partial^{|\alpha|}(\varphi_{\nu}f)(x_{i})(x-x_{i})
^{\alpha}|^{2}dx\leq
$$
$$
\Omega_{d}
\sum_{1\leq|\alpha|\leq n-1}(1/2)^{2|\alpha|+d}
|\partial^{|\alpha|}(\varphi_{\nu}f)(x_{i})|^{2}\rho^{2|\alpha|+d}\leq
$$
$$
C_{d,n}\sum_{|\gamma|\leq n+m-1}   \rho^{2|\gamma|}
\|\partial^{|\gamma|}(\varphi_{\nu}f)\|^{2}_{L_{2}(B(x_{i},\rho))}.
$$

Next, to estimate the third term in (4.4)
 we use the Schwartz inequality and the assumption $n>d/2$
$$
|\int_{0}^{\tau}t^{n-1}\partial
^{|\alpha|}(\varphi_{\nu}f)(x_{i}+t\vartheta)\vartheta^{\alpha}dt|^{2}
\leq
$$
$$
\left(\int_{0}^{\tau}t^{n-d/2-1/2}|t^{d/2-1/2}\partial
^{|\alpha|}(\varphi_{\nu}f)(x_{i}+t\vartheta)|dt\right)^{2}\leq
$$
$$
C_{d,n}\tau^{2n-d}\int_{0}^{\tau}t^{d-1}|\partial
^{|\alpha|}(\varphi_{\nu}f)(x_{i}+t\vartheta)|^{2}dt.
$$

We integrate both sides of this inequality over the ball
$B(x_{i},\rho/2)$ using the spherical coordinate system
$(\tau,\vartheta).$

$$
\int_{0}^{\rho/2}\tau^{d-1}\int_{0}^{2\pi}
|\int_{0}^{\tau}t^{n-1}\partial
^{|\alpha|}(\varphi_{\nu}f)(x_{i}+t\vartheta)\vartheta^{\alpha}dt|^{2}d\vartheta
d\tau\leq
$$
$$C_{d,n}\int_{0}^{\rho/2}t^{d-1}\left(\int_{0}^{2\pi}\int_{0}^{\rho/2}
\tau^{2n-d}|\partial^{|\alpha|}(\varphi_{\nu}f)(x_{i}+t\vartheta)|^{2}
\tau^{d-1}d\tau
d\vartheta\right)dt\leq
$$
$$
C_{d,n}\rho^{2n}\|\partial^{|\alpha|}
(\varphi_{\nu}f)\|^{2}_{L_{2}(B(x_{i},\rho))},
$$
where $\tau=\|x-x_{i}\|\leq\rho/2, |\alpha|=n.$

Finally, if $n>d/2$ and $k=n+m-1$,

$$\|\varphi_{\nu}f\|^{2}_{L_{2}(B(x_{i},\rho/2))}\leq
C_{d,k}\left(\rho^{d}|f(x_{i}|^{2}+
\sum_{j=1}^{k}\sum_{1\leq|\alpha|\leq j}\rho^{2|\alpha|}
\|\partial^{|\alpha|}
(\varphi_{\nu}f)\|^{2}_{L_{2}(B(x_{i},\rho))}\right),
$$
where $k>d-1$ since $n>d/2$ and $m>d/2.$
Since balls $B(x_{i},\rho/2)$ cover the manifold and the cover by
$B(x_{i},\rho)$
has a finite multiplicity $\leq R_{0}(M)$ the
summation over all balls gives
$$
\|f\|^{2}_{L_{2}(M)}\leq C(M,k)\left\{\rho^{d}\left(\sum_{i=1}^{\infty}
|f(x_{i})|^{2}\right)+\sum_{j=1}^{k}\rho^{2j}\|f\|^{2}_{H^{j}(M)}
\right\}, k>d-1.
$$
Using this inequality and the regularity theorem for Laplace-Beltrami
operator (see [9]) we obtain
$$
\|f\|_{L_{2}(M)}\leq$$
$$
C(M,k)\left\{\rho^{d/2}\left(\sum_{i=1}^{\infty}
|f(x_{i})|^{2}\right)^{1/2}+\sum_{j=1}^{k}\rho^{j}\left(\|f\|+
\|\Delta^{j/2}f\|\right)
\right\}, k>d-1.
$$

For the self-adjoint $\Delta$ for any $a>0,\rho>0, 0\leq j\leq k$ we have
the following interpolation inequality
$$
\rho^{j}\|\Delta^{j/2}f\|\leq a^{2k-j}\rho^{2k}\|\Delta^{k}f\|
+c_{k}a^{-j}\|f\|.
$$

Because in the last inequality we are free to choose any $a>0$ we are
coming to
our main claim.
\end{proof}

The next goal is to extend the last estimate to the Sobolev norm.

\begin{thm}
For any $k>d-1$
there exist constants $C(M,k)>0,\rho(M,k)>0,$ such that for any
$0<\rho<\rho(M,k)$, any admissible set $M_{\rho}=\{x_{i}\}$,
 any $m=2^{l}, l=0, 1, ...$, any smooth $f$ which is zero on
$M_{\rho}$ and any $t\geq 0$

\begin{equation}
\|\Delta^{t}f\|\leq \left(C(M,k)\rho^{2k}\right)^{m}\|\Delta^{km+t}
f\|, t\geq 0.
\end{equation}

\end{thm}

We will obtain this estimate as a consequence of the following Lemma.

\begin{lem}
1) If for some $f\in H^{2s}(M), a,s>0,$
\begin{equation}
\|f\|\leq a\|\Delta^{s}f\|,
\end{equation}
then for the same $f, a, s$ and all $t\geq 0, m=2^{l}, l=0, 1, ...,$
\begin{equation}
\|\Delta^{t}f\|\leq a^{m}\|\Delta^{ms+t}f\|,
\end{equation}
if $f\in H^{2(ms+t)}(M).$

\end{lem}

\begin{proof}
Let us remind that $\{\lambda_{j}\}$ is the set of eigen values of the
operator $\Delta$ and $\{\varphi_{j}\}$ is the set of corresponding
orthonormal eigen functions. Let $\{c_{j}=<f,\varphi_{j}>\}$ be the set of
Fourier
 coefficients of the function
$f$ with respect to the orthonormal basis $\{\varphi_{j}\}.$ Using the
Plancherel Theorem we can write our assumption (4.7) in the form

$$
\|f\|^{2}\leq a^{2}\left(\sum_{\lambda_{j}\leq a^{-1/s}}
\lambda_{j}^{2s}|c_{j}
|^{2}+\sum_{\lambda_{j}> a^{-1/s}}\lambda_{j}^{2s}|c_{j}
|^{2}\right).
$$
Since for the first sum $a^{2}\lambda_{j}^{2s}\leq 1$,

$$0\leq\sum_{\lambda_{j}\leq a^{-1/s}}(|c_{j}|^{2}-a^{2}\lambda_{j}^{2s}
|c_{j}|^{2})\leq \sum_{\lambda_{j}> a^{-1/s}}(a^{2}\lambda_{j}^{2s}
|c_{j}|^{2}-|c_{j}|^{2}).
$$

Multiplication of  this inequality by $a^{2}\lambda_{j}^{2s}$
 will only
 improve the existing inequality
and then using the Plancherel Theorem once again we will obtain

$$\|f\|\leq a\|\Delta^{s}f\|\leq a^{2}\|\Delta^{2s}f\|.
$$

It is now clear that using induction we can prove

$$
\|f\|\leq a^{m}\|\Delta^{ms}f\|, m=2^{l}, l\in \mathbb{N}.
$$

But then, using the same arguments we have for any $\tau>0$

$$
0\leq
\sum_{\lambda_{j}\leq a^{-1/s}}(a^{2\tau}\lambda_{j}^{2\tau s}
|c_{j}|^{2}
-a^{2(m+\tau)}\lambda_{j}^{2(m+\tau)s}|c_{j}|^{2})\leq
$$
$$\sum_{\lambda_{j}> a^{-1/s}}(a^{2(m+\tau)}\lambda_{j}^{2(m+\tau)s}
|c_{j}|^{2}-a^{2\tau}\lambda_{j}^{2\tau s}|c_{j}|^{2}),
$$
that gives the desired inequality (4.8) if $t=s\tau.$
\end{proof}

To prove (4.6) from the Theorem 4.2 it is enough to apply the last
Lemma 4.3 to the Lemma 4.1
 with $ a=C(M,k)\rho^{2k}$.

Next we are going to construct polyharmonic splines on manifolds.
We will need the following Lemma that gives an equivalent norm on
Sobolev spaces.
Recall that the norm in the Sobolev space were introduced in the
Introduction.

\begin{lem}
For any $k>d-1$ and any $\rho$-admissible set $M_{\rho}=\{x_{i}\},$
 the norm
of the Sobolev space $H^{2k}(M)$ is equivalent to the norm
\begin{equation}
\|\Delta^{k}f\|+\left (\sum_{x_{\gamma}\in
M_{\rho}}|f(x_{\gamma})|^{2}\right)^{1/2}.
\end{equation}
\end{lem}
The proof of the Lemma can be obtained as a consequence of the
Theorem 4.2, the Sobolev embedding Theorem and regularity of the
Laplace-Beltrami operator.

Given a $\rho$-admissible set $M_{\rho}, |M_{\rho}|=N,$ and a sequence
 of complex numbers $\{v_{\gamma}\}_{1}^{N}$ we
will be
 interested to find a
 function $s_{k}\in H^{2k}(M), k $ is large enough such that

a) $ s_{k}(x_{\gamma})=v_{\gamma}, x_{\gamma}\in \ M_{\rho};$

b) function $s_{k}$ minimizes functional $u\rightarrow \|\Delta^{k}u\|$.
\begin{lem}
The minimization problem has a unique solution if $k>d-1$.
\end{lem}
\begin{proof}

According to the last Lemma 4.4 it is enough to minimize the norm
(4.9). For the given sequence ${v_{\gamma} }$ consider a function
$f$ from $H^{2k}(M)$ such that $f(x_{\gamma})=v_{\gamma}.$ We
consider $H^{2k}(M)$ as the Hilbert space with the inner product
$$<f,g>=\sum_{x_{\gamma}\in
M_{\rho}}f(x_{\gamma})g(x_{\gamma})+ <\Delta^{k/2}f,\Delta^{k/2}g>.$$

Let $Pf$
 denote the orthogonal projection of the function $f$
 on the subspace
$U^{2k}(M_{\rho})=\left \{f\in H^{2k}(M)|f(x_{\gamma})=0\right \}$ with
respect to the new scalar product.
Then the function $g=f-Pf$ will be the unique solution to the
above minimization problem for the
 functional $u\rightarrow \|\Delta^{k} u\|, k>d-1.$
\end{proof}

We prove the Lemma 1.2 from the Introduction i.e. a function
$u\in H^{2k}(M)$ is a solution of the variational
problem 1)-2) if and only if it satisfies the following
equation in the sense of distributions
\begin{equation}
\Delta^{2k}u=\sum_{\nu=1}^{N} \alpha_{\nu}\delta(x_{\nu})
\end{equation}
where $\delta(x_{\nu})$ is the Dirac measure at $x_{\nu}$.

Indeed, we already know that for every solution $u$ of the above
variational
problem
\begin{equation}
0=<\Delta^{k}u,\Delta^{k}h>
=\int_{M}\Delta^{k}u\overline{\Delta^{k}h},
\end{equation}
where $h$ is any function which is zero on $M_{\rho}$.

Let $\{\xi_{\gamma}\}$ be the set of $C_{0}^{\infty}(M)$ functions
such that their supports are disjoint and
$\xi_{\nu}(x_{\mu})=\delta_{\mu\nu}$, where $\delta_{\nu\mu}$ is
the Kronecker delta. Then for any $\psi\in C_{0}^{\infty}(M)$ the
function

$$\psi-\sum_{\nu=1}^{N}\psi(x_{\nu})\xi_{\nu}
$$
is zero on $M_{\rho}$ and

$$
0=\int_{M}\Delta^{k}u\overline{
\Delta^{k}(\psi-\sum_{\nu=1}^{N}\psi(x_{\nu})\xi_{\nu})}=
$$
$$\int_{M}\Delta^{2k}u\overline{\psi}-
\sum_{\nu}^{N}<\Delta^{k}u,\Delta^{k}\xi_{\nu}>\overline{\psi(x_{\nu})}.
$$
In other words $\Delta^{2k}u$ is a distribution of the form

$$
\Delta^{2k}u=\sum\alpha_{\nu}\delta(x_{\nu}),
$$
where $\alpha_{\nu}=<\Delta^{k}u,\Delta^{k}\xi_{\nu}>.$

So every solution of the variational problem is a solution of (4.10).

Conversely, if $u$ is a solution of (4.10) then since the Dirac
measure belongs to the space $H_{-\varepsilon-d/2}(M), d=dim M,
\varepsilon>0,$ the Regularity Theorem for elliptic operator
$\Delta^{2k}$ of order $4k$ implies that $u\in H_{2k}(M)$ and for
any $h$ which is zero on $M_{\rho}$ we have

$$
<\Delta^{k}u,\Delta^{k}h>=<\Delta^{2k}u,h>=0,
$$
that shows that $u$ is a the solution for 1)-2).

Lemma 1.2 is proved.

Now  we can prove the
Approximation Theorem 1.4, which plays a key role in the proof of the
Theorem 1.3.

\begin{proof}

Using the Theorem 4.2 with $k=t=d$ and the continuous embedding
$H^{d}(M)\subset H^{s}(M), d\geq s,$ we obtain for every $s\leq d$

$$\|\Delta^{s}(s_{n}(f)-f)\|\leq C(M) \|\Delta^{d}(s_{n}(f)-f)\|\leq
\left(C(M)\rho^{2d}\right)^{m}\|\Delta^{d(m+1)}(s_{n}(f)-f)\|,
$$
where $n=d(m+1), m=2^{l}, l=0, 1, ....$
By minimization property we obtain

$$\|\Delta^{s}(s_{n}(f)-f)\|\leq
\left(C(M)\rho^{2}\right)^{n-d}\|\Delta^{n}f\|, n=d(2^{l}+1).
$$

If $f\in E_{\omega},$ i.e. $f$ is a linear combination of eigen functions
whose eigen
values belong to $[0, \omega]$, then $\|\Delta^{n}f\|\leq\omega^{n}\|f\|$,
and

$$\|\Delta^{s}(s_{n}(f)-f)\|\leq
\omega^{d}
\left(C(M)\rho^{2}\omega\right)^{n-d}\|f\|, n=d(2^{l}+1), s\leq d.
$$
To obtain corresponding estimates in the uniform norm it is enough to
combine the above inequalities with the Sobolev embedding Theorem.
The Approximation Theorem 1.4 is proved.
\end{proof}

\makeatletter \renewcommand{\@biblabel}[1]{\hfill#1.}\makeatother

\end{document}